\documentclass{article}
 \usepackage{amsmath, latexsym,amsthm, amsfonts, amssymb, amsxtra,amscd, enumerate,color, cancel}
\usepackage{float,wrapfig}
\usepackage[english]{babel}
\usepackage[usenames,dvipsnames,svgnames,table]{xcolor}
\usepackage{graphicx}
\RequirePackage[colorlinks,citecolor=blue,urlcolor=blue]{hyperref}

\newcommand{\e}{\textrm{e}}

\newtheorem{theorem}{Theorem}

\newtheorem{lemma}{Lemma}
\newtheorem{proposition}{Proposition}

\newtheorem*{example}{Example}

\newcommand{\F}{\mathcal{F}}
\newcommand{\FF}{(\F_t)_{t\geq 0}}
\newcommand{\R}{\mathbb{R}}

\newcommand{\Po}{\mathbb{P}}
\newcommand{\E}{\mathbb{E}}

\begin{document}
\title{Optimal stopping of oscillating Brownian motion}
\author{Ernesto  Mordecki\thanks{Universidad de la Rep\'ubica, Facultad de Ciencias, Centro de Matem\'atica, Igu\'a 4225, 11400, Montevideo, Uruguay, e-mail: mordecki@cmat.edu.uy}  
\ and\   Paavo Salminen\thanks{\AA bo Akademi University, Faculty of Science
  and Engineering, FIN-20500 \AA bo, Finland, e-mail: phsalmin@abo.fi.}}
\maketitle
\abstract{
We solve optimal stopping problems for an oscillating Brownian motion,
i.e. a diffusion with positive piecewise constant volatility 
changing at the point $x=0$. Let $\sigma_1$ and $\sigma_2$ denote the volatilities on the negative and positive half-lines, respectively.
Our main result is that continuation region of the optimal stopping problem 
with reward $((1+x)^+)^2$ is disconnected,
if and only if $\sigma_1^2<\sigma_2^2<2\sigma_1^2$.
Based on the fact that the skew Brownian motion in natural scale is an 
oscillating Brownian motion, the obtained results are translated into corresponding
results for the skew Brownian motion. 
}\\

\noindent{\bf AMS Subject Classification:} 60J60, 60J65, 62L15\\

\noindent{\bf Keywords:} excessive function, integral representation of excessive functions.

\section{Introduction}

The optimal stopping problems of diffusions with exceptional points has attracted interest in recent years. 
These include cases where the underlying diffusion has sticky points, 
skew points, or discontinuities in the diffusion coefficients.
One of the first findings is that in the presence of sticky points the classical smooth fit principle does not necessarily hold, even for differentiable payoff functions (as found in Crocce and Mordecki  \cite{CrocceMordecki}, and Salminen and Ta \cite{SalminenTa}). 
A second finding is that if the diffusion has a skew point, it can be the case that this point is in the continuation region for all discount values, as found by Alvarez and Salminen \cite{AlvarezSalminen} and Presman \cite{Presman}.
A third one is that the continuation region in these cases can be disconnected, 
as observed in \cite{AlvarezSalminen} for the skew Brownian motion, 
and also found recently by Mordecki and Salminen \cite{MordeckiSalminen} 
for a diffusion with discontinuous drift and payoff function is $(1+x)^+$.
General verification results for diffusions with discontinuous coefficients were obtained by R\"uschendorf and Urusov \cite{RuschendorfUrusov}.
An exposition of the general theory of optimal stopping (including historical comments)
can be found in Shiryaev \cite{Shiryaev} and Peskir and Shiryaev \cite{PeskirShiryaev}.

In this paper the focus is on the case when the underlying diffusion has discontinuous infinitesimal variance.
We then consider the optimal stopping problem for the oscillating Brownian motion (OBM),
a diffusion with positive piecewise constant volatility changing at the origin.
For details and further results on OBM, 
see Keilson and Wellner \cite{KeilsonWellner},
Lejay and Pigato \cite{LejayPigato}, and the references therein.
Our main results are the following: 
Firstly, for the payoff $(1+x)^+$ the solution of the optimal stopping
problem for the OBM is one sided for all values of the parameters, 
but for the payoff $((1+x)^+)^2$ the continuation region is disconnected for some values of the parameters.
Hence, this latter situation is similar as the one in \cite{MordeckiSalminen}.
Secondly, based on the fact that the skew Brownian motion (SBM) in natural scale is an OBM,
we obtain a result that connects the solutions of the respective optimal stopping problems for SBM and OBM,
finding that the non-differentiability of the scale function of SBM at the origin plays a key r\^ole in understanding
some of the phenomena that appear in the solutions of these problems.

\section{Diffusions and optimal stopping.}
Consider a conservative and regular one-dimensional (or linear) diffusion 
$X=(X_t)_{t\geq 0}$
taking values in $\R$,
in the sense of It\^o and McKean \cite{ItoMcKean} (see also Borodin and Salminen \cite{BorodinSalminen}). 
Let $\Po_x$ and $\E_x$ denote  the probability and the expectation associated with $X$ when starting from $x$, respectively;
$m$ denotes the speed measure and $S$ the scale function. 
For $r\geq 0$ let $\varphi_r$ ($\psi_r$) be the decreasing (increasing) positive fundamental solution of the generalized ODE 
\begin{equation}
\label{gen}
\frac d{dm}\frac d{dS} u=ru,
\end{equation}  
satisfying the appropriate boundary conditions (see \cite{BorodinSalminen} II.10 p.18). 
Denote by $\mathcal{M}$ the set of all stopping times in the filtration $\FF$, the usual augmentation of the natural 
filtration generated by $X.$ 
Given a continuous reward function $g\colon\R\to [0,\infty)$ and a discount factor ${r} \geq 0$, 
consider the optimal stopping problem consisting of finding a function $V_r$ and a stopping time $\tau^*\in\mathcal{M}$, such that
\begin{equation}\label{osp}
V_r(x)=\E_{x}[{\e^{-r\tau^*} g(X_{\tau^*})}] = \sup_{\tau \in \mathcal{M}}\E_{x}[{\e^{-r\tau} g(X_{\tau})}],
\end{equation}
where on the set $\{\tau=\infty\}$ 
$$
\e^{-r\tau} g(X_{\tau}):=\limsup_{t\to\infty}{\e^{-r t} g(X_{t})}.
$$
The \emph{value function} $V_r$ and the \emph{optimal stopping time} $\tau^*$  constitute  the solution of the problem.
The optimal stopping time $\tau^*$ in (\ref{osp}), 
can be characterized (see Theorem 3, Section 3.3 in \cite{Shiryaev}) as the first entrance time into the stopping region
\begin{equation}\label{eq:set}
\Gamma_r:=\{x\colon V_r(x)=g(x)\}.
\end{equation}
The set  ${\rm C}_r:=\R\setminus \Gamma_r$ is called the continuation region.

Our main tools to solve the optimal stopping problem for OBM
are the representation theory for excessive functions, and the following two results from the
theory of optimal stopping.  
The first one (Theorem \ref{smoothfit1}) --formulated here for a left boundary point of the stopping region-- is the smooth fit theorem, proof of which can be found in \cite{Salminen} or \cite{Peskir};
the second one (Proposition \ref{apu}) is a verification result, for the proof see Corollary on p. 124 in \cite{Shiryaev}.
\begin{theorem}
\label{smoothfit1}
 Let $z$ be a left boundary point of $\Gamma_r,$ i.e., 
$[z,z+\varepsilon_1)\subset \Gamma_r$ and $(z-\varepsilon_2,z)\subset
{\rm C}_r$ for some positive $\varepsilon_1$ and  $\varepsilon_2.$
Assume that the
reward function 
$g$ and the fundamental solutions $\varphi_r$
and $\psi_r$ are  
differentiable at $z.$ Then the value
function $V_r$ in (\ref{osp})  is 
differentiable at $z$ and it holds $V_r^\prime(z)=g^\prime(z).$
\end{theorem}
%
%
%
%
%
\begin{proposition}\label{apu}
Let $A\subset \R$ be a nonempty Borel subset of $\R$ and 
$$
H_A:=\inf\{t:X_t\in A\}.
$$ 
Assume that the function
$$
\widehat{V}(x):=\mathbb{E}_x\left[\textrm{e}^{-r\,H_A}g(X_{H_A})\right]
$$
is $r$-excessive and dominates $g$. Then $\widehat{V}$ coincides with the value function of OSP (\ref{osp}) and $H_A$ is an optimal stopping time.
\end{proposition}

\section{Oscillating Brownian motion}\label{section:obm}
Consider the diffusion satisfying the stochastic differential equation
\begin{equation*}
    X_t=x+\int_0^t \sigma(X_s)W_s,
\end{equation*}
where
$$
\sigma(x)=
\begin{cases}
{\sigma_1},&x<0,\\
{\sigma_2},&x\geq 0,
\end{cases}
$$
and  $(W_t)_{t\geq 0}$ is a standard Brownian motion. 
The diffusion $X$ is called an oscillating Brownian motion (OBM).
Notice that this process is in natural scale, i.e. the scale function is $S(x)=x$, and
the speed measure is
$$
m(dx)=
\begin{cases}
(2/\sigma_1^2)dx,\text{ $x<0$,}\\
(2/\sigma_2^2)dx,\text{ $x>0$.}\\
\end{cases}
$$
(by definition there is no mass at $x=0$).
Let
\begin{equation*}
\lambda^{\pm}_1=\pm{\sqrt{2r}\over \sigma_1},
\qquad
\lambda^{\pm}_2=\pm{\sqrt{2r}\over \sigma_2}.\\
\end{equation*}
The decreasing fundamental solution is
\begin{equation}\label{eq:phi}
\varphi_r(x)=
\begin{cases}
A_1\exp(\lambda_1^-x)+A_2\exp(\lambda_1^+x),&x<0,\\
\exp(\lambda_2^-x),& x\geq 0,\\
\end{cases}
\end{equation}
where the constants $A_1$ and $A_2$ are determined so that $\varphi_r$ is continuous and differentiable at 0. Hence,
\begin{align*}
A_1&={\lambda_1^+-\lambda_2^-\over \lambda_1^+-\lambda_1^-}=
{1+\sigma_1/\sigma_2\over 2},
\qquad
A_2={\lambda_2^- -\lambda_1^-\over \lambda_1^+-\lambda_1^-}
=
{1-\sigma_1/\sigma_2\over 2}.
\end{align*}
Analogously, the increasing solution is 
\begin{equation}\label{eq:psi}
\psi_r(x)=
\begin{cases}
\exp(\lambda_1^+x),& x<0,\\
B_1\exp(\lambda_2^+x)+B_2\exp(\lambda_2^-x),&x\geq 0,
\end{cases}
\end{equation}
with
$$
B_1={\lambda_1^+-\lambda_2^-\over \lambda_2^+-\lambda_2^-}=
{1+\sigma_2/\sigma_1\over 2},
\qquad
B_2={\lambda_2^+ -\lambda_1^+\over \lambda_2^+-\lambda_2^-}
=
{1-\sigma_2/\sigma_1\over 2}.
$$

\section{Optimal stopping of OBM}

We first analyze the optimal stopping problem \eqref{osp} for the diffusion introduced above and the reward function 
$$
g_1(x)=
\begin{cases}
0,& x\leq-1,\\
1+x, & x>-1.\\
\end{cases}
$$
The following result shows that the solution of this problem is one sided.


\begin{proposition} \label{prop0} 
Consider the OSP problem \eqref{osp} with payoff $g_1$.
For all values of $r>0$, $\sigma_1$ and $\sigma_2$ the continuation region is given by
$$
{\rm C}_r=(-\infty,c),
$$
where $c=c(r)>-1$ is the unique solution of the equation
\begin{equation}\label{eqH}
\psi_r'(x)(1+x)-\psi_r(x)=0.
\end{equation}
Furthermore
\begin{equation}\label{root}
2r \lesseqqgtr \sigma_1^2\quad \Rightarrow\quad c(r) \lesseqqgtr 0.
\end{equation}
\end{proposition} 
\begin{proof}
To prove the first statement, consider the functions (cf. \cite{Salminen})
\begin{align}\label{H-}
H_-(x)&:=
\psi_r'(x)g(x)-\psi_r(x)g'(x)=
\psi_r'(x)(1+x)-\psi_r(x),\\
\label{G+}
H_+(x)&:=
\varphi_r(x)g'(x)-\varphi_r'(x)g(x)=
\varphi_r(x)-\varphi_r'(x)(1+x),
\end{align} 
and their  derivatives for $x>-1$ and $x\not= 0$  can be expressed as
\begin{align*}
H_-'(x)&=m(x)\frac d{dm}H_-(x)=m(x)\psi_r(x)\left(r(1+x)-\frac d{dm}\frac d{dx} (1+x)\right)\\
&=m(x)\psi_r(x)r(1+x),
\end{align*}
and, similarly, 
\begin{equation*}\label{dH+}
H'_+(x)=-m(x)\varphi_r(x)r(1+x),
\end{equation*}
where it is used that  $\varphi_r$ and $\psi_r$ solve \eqref{gen}.
Observe now that the function $H_-$ in (\ref{H-}) has a unique positive root, 
since for $x>-1$ the derivative is strictly positive, 
$H_-(-1)=-\psi_r(-1)<0$, and $H_-(x)\to\infty$ as $x\to\infty$.
Therefore, equation \eqref{eqH} has a unique solution as claimed.
The rest of the proof is standard, see for instance \cite{Salminen} or the detailed proof of Proposition \ref{prop1} below.
Statement \eqref{root} follows since $H_-(0)=\sqrt{2r}/\sigma_1-1$.
\end{proof}

Next we study the OSP \eqref{osp} for OBM with $0<\sigma_1\leq\sigma_2$ 
and the reward function
\begin{equation}\label{q}
g(x)=
\begin{cases}
(1+x)^2, & x>-1,\\
0,& x\leq-1.
\end{cases}
\end{equation}
In this situation it is seen that, for some specific values of the parameters,
the continuation region is disconnected.
The approach to analyze this problem is similar as the one in \cite{MordeckiSalminen}.

Define now
\begin{align}\label{G-}
G_-(x):=&
\psi_r'(x)g(x)-\psi_r(x)g'(x)\notag\\
=&(1+x)\left(
\psi_r'(x)(1+x)-2\psi_r(x)\right),\\
\label{G+}
G_+(x):=&
\varphi_r(x)g'(x)-\varphi_r'(x)g(x)\notag\\
=&
2\varphi_r(x)(1+x)-\varphi_r'(x)(1+x)^2.
\end{align} 
These functions are used below to verify the excessivity of the proposed value function. 
The derivatives for $x>-1$ and $x\not= 0$ are 
\begin{align}
G'_-(x)&=
m(x)\psi_r(x)
\begin{cases}
r(1+x)^2-\sigma_1^2,& x<0,\\
r(1+x)^2-\sigma_2^2,& x>0,\\
\end{cases}
\label{dG-}
\\
G'_+(x)
&=m(x)\varphi_r(x)
\begin{cases}
\sigma_1^2-r(1+x)^2,& x<0,\\ 
\sigma^2_2-r(1+x)^2,& x>0.\\
\end{cases}
\label{dG+}
\end{align}
\begin{figure}[h]
\centering
\includegraphics[scale=0.25]{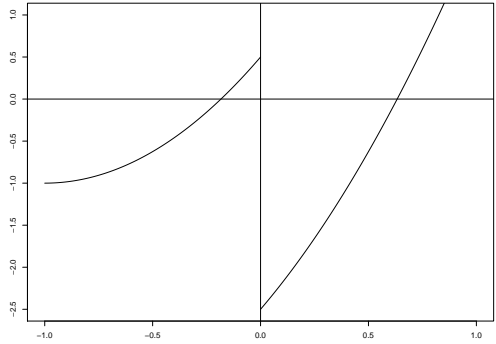}
\caption{The sign of the derivative $G_-'$ is ruled by the depicted above function 
$r(1+x)^2-\sigma(x)^2$.
Here the parameters are $r=1.5$,
 $\sigma_1=1$, $\sigma_2=2$.}
\label{figure:2}
\end{figure}

\begin{proposition} \label{prop1} In case $0<r\leq\sigma_1^2\leq\sigma_2^2$ the continuation region for  OSP
(\ref{osp}) is given by
$$
{\rm C}_r=(-\infty,c),
$$
where $c=c(r)$ is the unique positive solution of the equation
\begin{equation}\label{eqG}
\psi_r'(x)(1+x)-2\psi_r(x)=0,\quad x\geq -1.
\end{equation}
\end{proposition} 

\begin{proof} 
We show first that equation (\ref{eqG}) has a unique positive solution. For this consider for $x>-1$  the function 
$G_-$ defined in \eqref{G-}, the claim is equivalent with the statement that  $G_-$ has a unique positive zero. 
In fact, we claim a bit more; namely that  the function $G_-$ attains the global minimum at 
$x_0:=\sigma_2/\sqrt{r}-1>0,$ is negative and decreasing for $x\leq x_0,$ is increasing for $x>x_0,$ and has, therefore,  a unique zero. 
Analyzing $G'_-$ as given in \eqref{dG-}, 
 it is straightforward to deduce, since 
 $0<r\leq\sigma_1^2\leq\sigma_2^2,$ the claimed properties of $G_-.$ 
 
Let 
$$
H_c:=\inf\{t\, :\, X_t\geq c\},
$$
where $c$ is the unique solution of (\ref{eqG}), and define
\begin{equation}\label{hatV}
\widehat V(x):=\E_{x}\big[\textrm{e}^{-r H_c} g(X_{H_c})\big]=\begin{cases}\displaystyle{\frac{\psi_r(x)}{\psi_r(c)}}\,g(c),& x\leq c,\\ 
g(x) ,& x\geq c.\\
\end{cases}
\end{equation}
If $\widehat V$ is an $r$-excessive majorant of $g$ it follows from Proposition \ref{apu} that $\widehat V$ is the value function of OSP (\ref{osp}).  The excessivity can be checked with the method based on the representation theory of excessive functions (cf. \cite{Salminen} Section 3).   This boils down to study for $x\not= -1$ the functions 
\begin{align*}
I_V(x)&:=\psi'_r(x)\widehat V(x)-\psi_r(x)\widehat V'(x),\\
D_V(x)&:=\varphi_r(x)\widehat V'(x)-\widehat V(x)\varphi_r'(x).
\end{align*}
Clearly, $I_V(x)=0$ for $x\leq c$ and increasing for $x>c.$ Notice that $I_V=G_-$ on $[c,+\infty).$  Studying the derivative 
(w.r.t. the speed measure)  of $D_V$ it is easily seen that $D_V$ is positive and decreasing to 0 on $[c,+\infty).$ 
Consequently, $I_V$ and $D_V$ induce a (probability) measure which represents $\widehat V$ proving that   $\widehat V$ is $r$-excessive. To prove that $\widehat V$ is a majorant of $g$ consider for $-1<x<c$  
$$
\widehat V(x)\geq g(x) \quad \Leftrightarrow\quad \frac{\psi_r(x)}{g(x)}\geq\frac{\psi_r(c)}{g(c)}.
$$
The right hand side inequality holds since the derivative of $\psi_r(x)/g(x)$ is
 $G_-(x)$ which is negative for  $-1<x<c,$ as is shown above.
\end{proof}
If the volatilities are close enough, the problem is one sided for all discount values.
This is made precise in the next result.
\begin{proposition} \label{prop2} In case  $0\leq\sigma_1^2\leq\sigma_2^2\leq 2\sigma_1^2$ the continuation region for the OSP (\ref{osp})
is given by
$$
{\rm C}_r=(-\infty,c),
$$
where $c=c(r)$ is the unique solution of equation (\ref{eqG}).
As $r$ increases from $0$ to $+\infty$,  
$c(r)$ decreases monotonically  from $+\infty$ to $-1$. In particular, $c(r)=0$ for 
$r=2\sigma_1^2.$ 
\end{proposition} 
\begin{proof} If $r\leq\sigma_1^2$ the statement is the same as in Proposition \ref{prop1}. 
We assume next that $r\geq \sigma_2^2.$  The proof in this case is very similar to the proof of  Proposition \ref{prop1}. 
It can be seen that  $G_-$ attains the global minimum at 
$x_1:=\sigma_1/\sqrt{r}-1<0,$ is negative and decreasing for $x\leq x_1,$ is increasing for $x>x_1,$ and has, therefore,  a unique zero. Consequently, this root can be taken to be an optimal stopping point $c=c(r)$ and the analogous function $\widehat V$ as in (\ref{hatV}) can be proved to be the value of OSP (\ref{osp}). 
Finally, assume $\sigma_1^2<r<\sigma_2^2\leq2\sigma_1^2$  In this case, $G_-$ has a local maximum at 0, which is negative since 
\begin{equation}\label{gm0}
G_-(0)=\psi'_r(0)-2\psi_r(0)=\lambda^+_1-2={\sqrt{2r}\over \sigma_1}-2\leq 0. 
\end{equation}
Clearly, $G_-(0)=0$ (and then $c(r)=0$) when $r=2\sigma_1^2$.
Hence, equation (\ref{eqG}) has a unique positive root and the proof can be completed as in the previous cases.  
\end{proof}
\begin{proposition} \label{at0} 
Assume $0\leq\sigma_1^2\leq\sigma_2^2\leq 2\sigma_1^2$.
For  $r\geq 2\sigma_1^2$ there exist $A$ and $B$ such that the function
\begin{equation}\label{sfF}
F(x):=\begin{cases} A\exp(\lambda_1^+x)+B\exp(\lambda_1^-x),& x\leq 0,\\ (1+x)^2,& x\geq 0,\end{cases}
\end{equation}
satisfies the principle of smooth fit at 0, i.e., $F'(0-)=F'(0+)=2.$ 
The function $F$ is $r$-harmonic (and positive) on  $(-\infty,0)$ but not $r$-excessive if  $r<\sigma_2^2.$ 
For $r< 2\sigma_1^2$ the coefficient $B$ is negative and the function $F(x)\to -\infty$ as $x\to -\infty$ (and the function is not  $r$-excessive).
\end{proposition}
\begin{proof} We study only the case $r=r_0:=2\sigma_1^2$ and leave the details of the other cases to the reader. 
In this case $\lambda^+_1 ={\sqrt{2r}/\sigma_1}=1,$ and, obviously, 
$$ 
F(x):=\begin{cases} 2\,\e^x,& x\leq 0,\\ (1+x)^2,& x\geq 0,\end{cases}
$$
satisfies smooth fit at 0. 
 Consequently, $F$ is $r_0$-harmonic (and positive) on $(-\infty,0)$ and  it remains to prove that $F$ is not $r_0$-excessive. For this,  consider the representing function (this corresponds $G_-$ in (\ref{G-}))
$$
x\mapsto \psi_{r_0}'(x)F(x)-\psi_{r_0}(x)F'(x).
$$
The claim is that this function is not non-decreasing. 
Indeed, differentiate w.r.t. the speed measure to obtain
\begin{align*}
&\frac{d}{dm}\Big(\psi_{r_0}'(x)F(x)-\psi_{r_0}(x)F'(x)\Big)\\
&\hskip3cm=F(x)\frac{d}{dm}\frac{d}{dx}\psi_{r_0}(x)-\psi_{r_0}(x)\frac{d}{dm}\frac{d}{dx}F(x)\\
&\hskip3cm=\psi_{r_0}(x)\begin{cases} 0,& x<0,\\
{r_0}(1+x)^2-\sigma_2^2,& x>0.\\
\end{cases}
\end{align*}
Since $r_0=2\sigma_1^2<\sigma_2^2$ this derivative is negative, e.g.,  for small positive $x$-values; therefore,  $F$ is not ${r_0}$-excessive.   
 \end{proof}

For the theorem to follow, which can be seen as our main result concerning OSP (\ref{osp}), we need the following technical result.
\begin{lemma}\label{apu2} 
Consider a family $\{h_r\colon\R\to[0,\infty)\,;\, r\in I\}$ such that for each $r\in I\subset\R$ the function $h_r$ is $r$-excessive.
Assume that this family is dominated 
by a function $\widehat{h}$ (i.e. $h_r\leq \widehat{h}$) such that $\E_x(\widehat{h}(X_t))<\infty$ for all  $t\geq 0$ and $x\in\R$.  
Then, if for some $r_0$ the limit
$$
\lim_{r\to r_0}h_r(x)\to h_0(x),
$$
exists for all $x\in\R$, the function $h_0$ is $r_0$-excessive.
\end{lemma}
\begin{proof}
Consider
\begin{align*}
\mathbb{E}_x\left[\textrm{e}^{-r_0t}h_{0}(X_t)\right]&=\mathbb{E}_x\big[\lim_{r\to r_0}{\textrm{e}}^{-rt}h_{r}(X_t)\big]
=\lim_{r\to r_0}\mathbb{E}_x\left[{\textrm{e}}^{-rt}h_{r}(X_t)\right]\\ 
&\leq\lim_{r\to r_0}h_r(x) 
=h_{0}(x),
\end{align*}
where in the second step we use the dominated convergence theorem which is applicable since 
${\textrm{e}}^{-rt}h_r(X_t)\leq \widehat{h}(X_t)$.
\end{proof}
The theorem below states  that if  $r\in(2\sigma_1^2, \sigma_2^2)$ but is close enough to $\sigma_2^2$ 
then ${\rm C}_r$ has a bubble (i.e. an isolated bounded interval in the continuation region). 
However, the bubble disappears when  $r$ becomes larger than $\sigma_2^2$ 
or tends to $2\sigma_1^2$.  

\begin{theorem} \label{prop3} 
In case  $0<\sigma_1^2<2\sigma_1^2<\sigma_2^2$ there exists $r_0\in(2\sigma_1^2,\sigma_2^2)$ with the following properties:  
\begin{description}
 \item{\rm(a)}\hskip.2 cm  If $r\in[r_0,\sigma_2^2)$ the continuation region is given by
 $$
{\rm C}_r=(-\infty,c_1)\cup (c_2,c_3),
$$
where $c_i=c_i(r),\, i=1,2,3,$ are such that 
$-1<c_1\leq c_2\leq 0<c_3$.
In particular, for $r=r_0$ it holds $c_1=c_2<0.$
\item{\rm(b)}\hskip.2 cm If $r\geq \sigma_2^2$ the continuation region is explicitly given by 
 $$
{\rm C}_r=(-\infty,c_-),
$$
where 
$$
c_-=c_-(r)={2\sigma_1\over\sqrt{2r}}-1<0,
$$ 
i.e. $c_-$ is the unique solution of (\ref{eqG}). 
\item{\rm(c)}\hskip.2 cm If $r<r_0$ the continuation region is given by 
 $$
{\rm C}_r=(-\infty,c_+),
$$
where $c_+=c_+(r)>0$ is the unique solution of (\ref{eqG}).
 \end{description}

\end{theorem} 

\begin{proof} The proof of  (b) is as the proof of Proposition \ref{prop2} when $r\geq \sigma^2_2$. 
Notice, however, that in the present case $c(r)<0$ for all $r\geq \sigma_2^2$.  

We consider next (c) in case $r\leq 2\sigma_1^2$. 
Studying $G'_-$ and $G_-(0)$ in \eqref{dG-} and \eqref{gm0} respectively, 
it is seen, 
as in the proof of Proposition \ref{prop1}, 
that equation $G_-(x)=0$ has for $r<2\sigma_1^2$ one (and only one) root $\rho=\rho(r)>\sigma_2/\sqrt{r}-1>0$.
In case $r=2\sigma_1^2$ there are two roots 
$\rho_1=0$ and 
$\rho_2>\sigma_2/(\sqrt{2}\sigma_1)-1>0.$ Proceeding as in the proof of Proposition \ref{prop1} it is seen that the stopping region is as claimed with $c_+=\rho$ if $r<2\sigma_1^2$ and $c_+=\rho_2$ if $r=2\sigma_1^2$.  

Finally we consider (a).
Assume now that there does not exist a bubble for any  
$r\in[2\sigma_1^2,\sigma_2^2].$  
Then for all $r\in[2\sigma_1^2,\sigma_2^2]$ we can find  $c=c(r)$ such that $\Gamma_r=[c,+\infty).$ 
Knowing that $c(r) >0$ for $r=2\sigma_1^2$ and $c(r)<0$ for $r=\sigma_2^2$ we remark first there does not exists $r$ such that $c(r)=0.$ Indeed, by Theorem \ref{smoothfit1}, the value should satisfy the smooth fit principle at 0 but from Proposition \ref{at0}  we know  that such functions are not $r$-excessive. 
Next, using  $\Gamma_{r_1}\subseteq \Gamma_{r_2}$ for $r_1<r_2$  (cf. Proposition 1 in \cite{MordeckiSalminen})  it is seen that $r\mapsto c(r)$ is non-increasing, and has, hence, left and right limits. Consquently, there exists a unique point 
$r_0$  such that 
$$
\hat c_+:=\lim_{r\uparrow r_0}c(r)>0 \quad {\rm and}\quad \hat c_-:=\lim_{r\downarrow r_0}c(r)<0.
$$ 
Under the assumption that there is no bubble the value function is of the form given in (\ref{hatV}), i.e., 
\begin{align}\label{hatVr} 
V_r(x)&=\begin{cases}\psi_r(x)\frac{(1+c(r))^2}{\psi_r(c(r))},& x\leq c(r),\\ 
(1+x)^2 ,& x\geq c(r).\\
\end{cases}\\
\nonumber
&=\E_x\big(\textrm{e}^{-rH_c}\, (1+X_{H_c})^2\big),
\end{align}  
where $H_c:=\inf\{y\,:\,X_t\geq c(r)\}.$    
For $\widehat{r}$ small enough, there exists an excessive majorant $h_{\widehat{r}}$, 
so that we can apply Proposition \ref{apu2},
since $\E_x \textrm{e}^{-\widehat{r}t}h_{\widehat{r}}(X_t)\leq \widehat{h}(x)<\infty$.
Then, letting in (\ref{hatVr})  $r\uparrow r_0$ yields an $r_0$-excessive function which by Proposition  \ref{apu} is the value of the corresponding OSP (\ref{osp}). Similarly, letting $r\downarrow r_0$ yields an  $r_0$-excessive function which should also be the value of the same OSP. However, the functions are clearly different and since the value is unique we have reached a contradiction showing that there exists at least one bubble, i.e. a bounded open interval $(x_1,x_2)\subseteq {\rm C}_r$ 
with endpoints $x_1$ and $x_2$ in the stopping set (see \cite{MordeckiSalminen}).
Proceeding similarly as in Proposition 6 in \cite{MordeckiSalminen},
it can be seen that if there is a bubble, there is at most one bubble, and it contains the origin, completing the proof.      
\end{proof}

We end this section by considering the case $\sigma_1^2\geq\sigma_2^2$ 
with quadratic reward \eqref{q},
and show that the stopping region in this case is always one sided.
\begin{proposition}
\label{prop00}
Consider the OSP problem \eqref{osp} for OBM with $\sigma_1^2\geq\sigma_2^2$, $r>0$ and $g(x)$ in \eqref{q}. 
For all values of $r>0$, the continuation region for  OSP
(\ref{osp}) is given by
$$
{\rm C}_r=(-\infty,c),
$$
where $c=c(r)>-1$ is the unique solution of the equation \eqref{eqG}.
Furthermore
$$
r\lesseqqgtr 2\sigma_1^2\quad\Rightarrow\quad c(r)\lesseqqgtr 0.
$$ 
\end{proposition} 
\begin{proof}
The proof follows the general lines of \cite{Salminen} developed in the previous proofs of Propositions \ref{prop0}
and \ref{prop1}.
Consider then the functions defined in \eqref{G-} and \eqref{G+}, with their respective derivatives in \eqref{dG-} and \eqref{dG+}.
As $\sigma_1^2\geq\sigma_2^2$, the derivative $G_-'(x)$ changes sign only once, from negative to positive.
Hence, equation \eqref{eqG} has only one root $c(r)>-1$, and the proof of (a) follows as claimed.
For (b), notice that the root $c(r)=0$ when $r=2\sigma_1^2$, the result follows from the monotonicity of the function $c(r)$.
This concludes the proof.
\end{proof}

\section{OSP for skew Brownian motion}

Consider a SBM  $(\widehat{X}_t)_{t\geq 0}$ with index $\beta\in(0,1)$
starting at $\widehat{x}\in\R$ (see \cite{ItoMcKean},  \cite{Walsh}, \cite{Lejay}).
This diffusion can be characterized by scale function
$$
\widehat{S}(x)=
\begin{cases}
x/(2(1-\beta)),& x< 0,\\
x/(2\beta),& x\geq 0,\\
\end{cases}
$$
and speed measure
$$
m(dx)=
\begin{cases}
4(1-\beta)dx,& x< 0,\\
4\beta dx,& x\geq 0.\\
\end{cases}
$$
It is known (see e.g. \cite{LejayPigato}) that $(\widehat{S}(\widehat{X}_t))_{t\geq 0}\overset{d}{=}(X_t)_{t\geq 0}$, i.e. the composition of
SBM with its scale function has the same law as OBM with $\sigma_1=1/(2(1-\beta))$ and $\sigma_2=1/(2\beta)$
and starting point $x=\widehat{S}(\widehat{x})$.
In other words, we can say that SBM in natural scale is OBM. 
We use this relationship to obtain conclusions about the OSP problem \eqref{osp}.
Given a payoff function $g\colon\R\to[0,\infty)$ we introduce the payoff function
\begin{equation}\label{ghat}
\widehat{g}(x)=(g\circ \widehat{S})(x).
\end{equation}
Due to the fact that $\widehat{S}$ is not differentiable at the origin, 
both functions $g$ and $\widehat{g}$ can not be differentiable at the origin.
The next result connects the optimal stopping problems for OBM and SBM.
\begin{proposition}
For $\beta\in(0,1)$ the optimal stopping problem \eqref{osp} for OBM $X$
with parameters $\sigma_1=1/(2(1-\beta))$, $\sigma_2=1/(2\beta)$ and continuous reward $g$
has value function $V(x)$ and stopping region $\Gamma$ if and only if
the optimal stopping problem for SBM $\widehat{X}$ with index $\beta\in(0,1)$
and reward $\widehat{g}$ in \eqref{ghat}
has stopping region $\widehat{\Gamma}=\widehat{S}^{-1}(\Gamma)$ 
and value function $\widehat{V}(y)=V(\widehat{S}(y))$.
\end{proposition}
\begin{proof} It holds
\begin{align*}
V(x)&=\sup_{\tau}\E_x\left(\textrm{e}^{-r\tau}g(X_{\tau})\right)=\sup_{\tau}\E_x\left(\textrm{e}^{-r\tau}\widehat{g}(\widehat{S}^{-1}(X_{\tau}))\right)\\
&=\sup_{\tau}\E_{\widehat{S}^{-1}(x)}\left(\textrm{e}^{-r\tau}\widehat{g}(\widehat{X}_{\tau})\right)=\widehat{V}(\widehat{S}^{-1}(x)),
\end{align*}
and this yields $\widehat{V}(y)=V(\widehat{S}(y))$.
Let
$
\widehat{\Gamma}=\{y\colon \widehat{g}(y)=\widehat{V}(y)\},
$
and consider
\begin{align*}
\widehat{S}(\widehat{\Gamma})
&=\{x\colon \exists y\in\widehat{\Gamma} \text{ such that } x=\widehat{S}(y)\}\\
&=\{x\colon \widehat{g}(\widehat{S}^{-1}(x))=\widehat{V}(\widehat{S}^{-1}(x))\}\\
&=\{x\colon g(x)=V(x)\}=\Gamma.
\end{align*}
This concludes the proof.
\end{proof}
From this proposition it follows that $0\notin\Gamma$ if and only if $0\notin\widehat{\Gamma}$,
because $\widehat{S}(0)=0$. Furthermore, 
$\Gamma$ is disconnected if and only if $\widehat{\Gamma}$ is disconnected,
as the function $\widehat{S}$ is strictly increasing and continuous.

\begin{example}
{\rm  Consider the problem \eqref{osp} for SBM with $\beta>1/2$ and reward $\widehat{g}(x)=(1+x)^+$ (cf. \cite{AlvarezSalminen}).
The corresponding OSP for OBM has $\sigma_1=1/(2(1-\beta))$, $\sigma_2=1/(2\beta)$,
and reward
\begin{equation}\label{gg}
g(x)=
\begin{cases}
(1+2(1-\beta)x)^+,& x<0,\\
1+2\beta x,&x\geq 0.
\end{cases}
\end{equation}
Notice that $g'(0-)=2(1-\beta)<2\beta=g'(0+)$ (see Figure \ref{3}).
As the scale function of OBM is $S(x)=x$, any $r$-excessive function $h$ satisfies (see p. 93 in \cite{Salminen})
\begin{equation}\label{ex}
h'(x-)\geq h'(x+),\text{ for all $x\in\R$}.
\end{equation} 
If $0\in\Gamma$ for some $r\geq 0$, then $V(0)=g(0)$, hence $V'(0-)<V'(0+)$ violating condition \eqref{ex}. 
We conclude that $0\notin\Gamma$ for any value of $r$, 
hence $0\notin\widehat{\Gamma}$ for any value of $r$ for the SBM problem.
This is a particular case of the result obtained in Proposition 1 in \cite{AlvarezSalminen}.
\begin{figure}[H]
\centering
\includegraphics[scale=0.25]{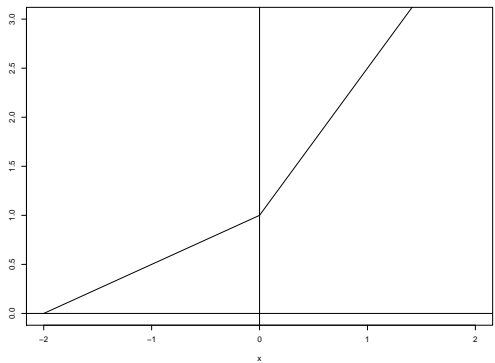}
\caption{The reward $g$ in \eqref{gg} for OBM, with $\sigma_1=2$ and $\sigma_2=2/3$ ($\beta=3/4$).}
\label{3}
\end{figure}
}
\end{example} 

{\bf Acknowledgement.} This research has been partially supported by grants from Magnus Ehrnrooths stiftelse, Finland.
The authors acknowledge also the hospitality of \AA bo Akademi University, Turku-\AA bo, Finland and Universidad de la Rep\'ublica,
Montevideo, Uruguay.

\end{document}